\documentclass[12pt,leqno,oneside]{amsart}

\usepackage{amsmath,amsthm,amssymb,amsfonts,ifpdf,xcomment}

\ifpdf
\usepackage[pdftex,pdfstartview=FitH,pdfborderstyle={/S/B/W 1},%
colorlinks=true, linkcolor=blue, urlcolor=blue, citecolor=blue,%
pagebackref=true]{hyperref}
\else
\usepackage[hypertex]{hyperref}
\fi

\usepackage[abbrev,msc-links,nobysame]{amsrefs}

\newtheorem {thm}   {Theorem}
\newtheorem* {thm*}   {Theorem}
\newtheorem* {prp*}   {Proposition}

\newtheorem {que}      [thm]    {Question}

\newtheorem {prp}[thm]  {Proposition}

\newtheorem {rmk} [thm]    {Remark}
\newcounter{AbcT}

\theoremstyle{definition}

\renewcommand{\d}{\delta}

\renewcommand{\l}{\lambda}

\newcommand{\R}{{\bf R}}

\newcommand{\Q}{{\bf Q}}

\newcommand{\Z}{{\bf Z}}

\newcommand {\cP} {{\mathcal P}}

\newcommand {\cQ} {{\mathcal Q}}

\newcommand{\wt}{\widetilde}

\title[Bernoulli convolutions]%
{On the dimension of Bernoulli convolutions for all transcendental parameters}
\author{P\'eter P. Varj\'u}
\thanks{
I gratefully acknowledge support
from the Royal Society
and support from ERC Grant no. 803711 `EFMA'}


\dedicatory{Dedicated to the memory of Jean Bourgain}

\begin{document}

\begin{abstract}
The Bernoulli convolution $\nu_\lambda$ with parameter $\lambda\in(0,1)$ is the probability measure
supported on $\R$ that is the law of the random variable $\sum\pm\l^n$, where the $\pm$ are independent
fair coin-tosses.
We prove that $\dim\nu_\lambda=1$ for all transcendental $\lambda\in(1/2,1)$.
\end{abstract}

\maketitle

Fix a number $\l\in(0,1)$ and let $X_0,X_1,\ldots$ be a sequence of independent random variables, which
take the values $\pm1$ with equal probability.
The Bernoulli convolution $\nu_\lambda$ with parameter $\lambda$ is the probability measure on $\R$
that is the law of the random variable of
\[
\sum_{n=0}^{\infty}X_n\lambda^n.
\]

Bernoulli convolutions are one of the most studied examples of self-similar measures, and they are objects
of great interest in fractal geometry.
The following two basic questions about them are still open.
\begin{que}\label{qu:ac}
For which values of $\lambda$ is $\nu_\l$ absolutely continuous with respect to the Lebesgue measure?
\end{que}
\begin{que}\label{qu:dim}
For which values of $\lambda$ does $\dim \nu_\l=1$ hold?
\end{que}

The latter question requires some explanation.
Feng and Hu \cite{feng-hu} proved that Bernoulli convolutions are exact dimensional, that is to say, there is number,
which we denote by $\dim\nu_\l$, such that
\[
\lim_{r\to 0} \frac{\log\nu_\l([x-r,x+r])}{\log r}=\dim\nu_\l
\]
for $\nu_\l$-almost every $x$.
It is worth noting that $\dim\nu_\l$ is also the infimum of the Hausdorff dimensions of all Borel subsets of $\R$
that have positive $\nu_\l$-measure, see \cite{Fal-techniques}*{Proposition 10.2} for a proof of this fact, which
holds for all exact dimensional measures.

For $\l<1/2$, the Bernoulli convolution $\nu_\l$ is the Cantor-Lebesgue measure on a Cantor set, it is singular
and has dimension $\log 2/\log\l^{-1}$.
For $\l>1/2$ the above questions are still not completely understood.
The study of Question \ref{qu:ac} goes back to Erd\H os \cite{erdos39}, who exhibited the first and only known examples
of parameters $\lambda\in(1/2,1)$ that make the Bernoulli convolution singular.
These are the inverses of Pisot numbers.
Erd\H os \cite{erdos} also proved that there is a number $a<1$ such that $\nu_\l$ is absolutely continuous for almost all
$\l\in(a,1)$.
Solomyak proved in the landmark paper \cite{Sol-Bernoulli} that one can take $a=1/2$ in the above statement.

In this paper, we make progress on the second question.
The above mentioned result of Solomyak implies that the set
\[
E:=\{\lambda\in(1/2,1):\dim\nu_\l<1\}
\]
is of $0$ Lebesgue measure.
Hochman proved in another landmark paper \cite{hochman} that the set $E$ is of $0$ packing and Hausdorff dimension.
We take this a step further and show that $E$ is countable.

\begin{thm}\label{th:main}
We have $\dim\nu_\l=1$ for all transcendental $\l\in(1/2,1)$.
\end{thm}

The dimension of Bernoulli convolutions for algebraic parameters is not fully understood.
The only known examples of parameters $\l\in(1/2,1)$ with $\dim(\nu_\l)<1$ are the inverses of Pisot numbers;
this fact can be traced back to Garsia \cite{Garsia-entropy} in some form.
Hochman \cite{hochman} expressed the dimension for algebraic parameters in terms of the so-called Garsia entropy of $\l$,
a quantity that have been studied recently in \cites{BV-entropy,AFKP-Bernoulli}.
We will briefly recall these results in Section \ref{sc:BV1}.

There is a folklore conjecture predicting that the dimension of self-similar measures equal their similarity dimension unless
exact overlaps occur.
See \cite{hochman}*{Sections 1.1 and 1.2}, where this conjecture is discussed.
(See also \cite{PS-problems}*{Question 2.6} for a closely related question about self-similar sets due to Simon.)
Theorem \ref{th:main} together with the results of Hochman \cite{hochman}*{Theorem 1.5} for algebraic parameters
establish this conjecture for Bernoulli convolutions.

It is not our aim to give a thorough discussion of the rich literature on Bernoulli convolutions.
Instead, we refer the interested readers to the surveys \cites{60y,Sol-survey,Var-ECM,Gou-Bernoulli-survey}.
The reader interested in Question \ref{qu:ac} is also recommended to consult the recent papers
\cites{shmerkin,Shm-Bernoulli-Lq,Var-Bernoulli-algebraic}.

Hochman \cite{hochman} proved Theorem \ref{th:main} conditionally on the hypothesis that there is a number $C>0$ such that
any two algebraic numbers $\xi_1$ and $\xi_2$ that are roots of (not necessarily the same) polynomials of degree
at most $n$ with coefficients $-1$, $0$ and $1$ satisfy the separation condition $|\xi_1-\xi_2|>C^{-n}$.
This hypothesis is very plausible, because there are at most $n3^{n+1}$ such algebraic numbers,
however, we do not know how to prove it.

On the other hand, we are able to prove the following weaker property.
Let $\xi\in(1/2,1/2^{2/3})$ be an algebraic number of degree at most $d$ with Mahler measure at most $M$
(see Section \ref{sc:Garsia} for the definition).
Then there are numbers $C_1,C_2,C_3$ depending only on $M$ such that the following holds.
Let $n>C_1d\log d$ be an integer and let $P$ be a polynomial of degree at most $n$
with coefficients $-1$, $0$ and $1$.
Then $|P(\l)|>C_2^{-n}$ for all $\lambda$ that satisfy $C_3^{-n}\le|\lambda-\xi|\le C_3^{-n+1}$.

The proof of Theorem \ref{th:main} relies heavily on the work of several mathematicians.
The statement in the previous paragraph is proved using an estimate on the values polynomials with small coefficients
evaluated on algebraic numbers, which was first used in the context of Bernoulli convolutions by Garsia \cite{Garsia-arithmetic},
together with the transversality property of these polynomials proved by Solomyak \cite{Sol-Bernoulli},
(see also Peres and Solomyak \cite{PS-Bernoulli}).
The algebraic number $\xi$ in the statement is found using a characterization of parameters with $\dim\nu_\l<1$
by Breuillard and Varj\'u \cite{BV-transcendent}.
The Mahler measure of $\xi$ is estimated using another paper \cite{BV-entropy} of the same authors.
Then the conclusion of the above statement is plugged into a result of Hochman \cite{hochman} to prove $\dim\nu_\l=1$.
 
The paper is organized as follows.
In Section \ref{sc:prelim}, we recall some facts from the above mentioned five papers.
We then prove Theorem \ref{th:main} in Section \ref{sc:proof} in just a few strokes.

\subsection*{Acknowledgment}
I am grateful to Emmanuel Breuillard, Vesselin Dimitrov and Ariel Rapaport for inspiring discussions and for carefully reading the manuscript.
I am grateful to Boris Solomyak and Pablo Shmerkin for helpful comments, which improved the presentation of this paper.
I am also grateful to the anonymous referee for suggesting Remark \ref{rm:referee}
and for other helpful comments, which improved the presentation of the paper.

\section{Preliminaries}\label{sc:prelim}

We recall some facts from the literature in this section, which will be used in the proof of Theorem \ref{th:main}.

\subsection{Estimates for the values of polynomials}\label{sc:Garsia}

Let $\l$ be an algebraic number with minimal polynomial
\[
a_dx^d+\ldots+a_1x+a_0=a_d(x-\lambda_1)\cdots(x-\lambda_d).
\]
The Mahler measure of $\lambda$ is defined as
\[
M(\l)=|a_d|\prod_{j:|\l_j|>1}|\l_j|.
\]
For more on the basic properties of this quantity we refer to \cite{BG-heights}*{Chapter 1}.

The utility of the following simple lemma (or a variant of it, rather, see \cite{Garsia-arithmetic}*{Lemma 1.51})
in the study of Bernoulli convolutions was first pointed out by Garsia.

\begin{thm}\label{th:Garsia}
Let $P$ be a polynomial of degree at most $n$ with coefficients $-1$, $0$ and $1$
and let $\l$ be an algebraic number of degree $d$.
Suppose that $P(\l)\neq0$.
Then
\[
|P(\l)|\ge M(\l)^{-n}(n+1)^{1-d}.
\]
\end{thm}

We adapt the proof from \cite{Garsia-arithmetic}*{Lemma 1.51}.

\begin{proof}
Let $a_d$ and $\l=\l_1,\ldots,\l_d$ be as above.
Then
\[
a_d^n\prod_{j=1}^d P(\l_j)\in\Z.
\]
Indeed, it is clearly in $\Q$, since it is invariant under all automorphisms of $\overline\Q$.
To show that it is an integer, it is enough to show that
\[
\Big|a_d^n\prod_{j=1}^d P(\l_j)\Big|_v\le 1
\]
for all finite places $v$ of $K(\l_1,\ldots,\l_d)$.\footnote{
Indeed, if a number $a\in\Q$ is not an integer, then the denominator of $a$ is divisible by a prime
$p$, hence $|a|_p>1$ and we also have $|a|_v>1$ for all places $v$ of $K(\l_1,\ldots,\l_d)$
that lie above $p$.
For basic properties of absolute values, we refer to \cite{BG-heights}*{Sections 1.2 and 1.3}.
}
Using that $|\cdot|_v$ is an ultrametric, we have
\[
\Big|a_d^n\prod_{j=1}^d P(\l_j)\Big|_v\le|a_d|_v^n\prod_{j=1}^d\max(|\l_j|_v^n,1).
\]
Applying Gauss's lemma (see e.g. \cite{BG-heights}*{Lemma 1.6.3}) for the product of polynomials $\prod_j(x-\l_j)$, we get
\[
\prod_{j=1}^d\max(|\l_j|_v,1)=|a_d|_v^{-1},
\]
which combined with our previous inequality gives the claim.

For each $j$, we can write
\[
|P(\l_j)|\le(n+1)\max(1,|\l_j|)^n.
\]
Using
\[
\Big|a_d^n\prod_{j=1}^d P(\l_j)\Big|\ge1,
\]
we get
\begin{align*}
|P(\l_1)|\ge& a_d^{-n}\prod_{j=2}^d |P(\l_j)|^{-1}\\
\ge&(n+1)^{1-d}a_d^{-n}\prod_{j=2}^d \max(1,|\l_j|)^{-n}\\
\ge&(n+1)^{1-d}M(\l)^{-n}.
\end{align*}
\end{proof}

\subsection{Transversality}\label{sc:Solomyak}

It is clear that a polynomial with coefficients $-1$, $0$ and $1$ cannot have a root in the interval $(0,1/2)$.
It is natural to expect that there is a larger interval $(0,a)$ with $a>1/2$, where such a polynomial may
have at most one root.
This is indeed true and was established by Solomyak \cite{Sol-Bernoulli} building on ideas from \cite{PS-transversality}.
In fact, a slightly stronger property called transversality also holds, which we recall now.

Let $A\subset\Z_{>0}$.
We write $\cP_A$ for the set of power series of the form
\[
1+\sum_{n\in A}a_nx^n,
\]
where $a_n\in\{-1,0,1\}$.
Let $x_0,\d>0$ be numbers.
We say that the interval $[0,x_0]$ is an interval of $\d$-transversality for $\cP_A$
if for all $x\in[0,x_0]$ and $f\in\cP_A$ the inequality $f(x)<\d$ implies $f'(x)<-\d$.

We note a consequence of this definition.
If $f(x_1)<\d$ for some $x_1\in[0,x_0]$, then $f(x)<\d-\d(x-x_1)$ for all $x\in[x_1,x_0]$.
Indeed, if this was not true, by continuity, there is a smallest $x\in[x_1,x_0]$ with $f(x)\ge \d-\d(x-x_1)$.
Clearly $f(t)<\d$ for all $t\in[x_1,x]$, which implies $f'(t)<-\d$ by the definition of transversality, which
contradicts the mean value theorem.
This means, in particular, that a function $f\in\cP_A$ may have at most one zero in an interval of $\d$-transverasilty.

There are three sets, which we will use in the role of $A$ in this paper, and now we introduce short notation for them.
We write $\cP:=\cP_{\Z_{> 0}}$ and
\[
\cP_i:=\cP_{\{n\in\Z_{> 0}:3\nmid n-i\}}
\]
for $i=1,2$.

We recall the following result from \cites{Sol-Bernoulli,PS-Bernoulli}.

\begin{thm}\label{th:Solomyak}
There is an absolute constant $\d>0$ such that $[0,2^{-2/3}]$ is an interval of $\d$-transversality for $\cP$
and $[0,2^{-1/2}]$ is an interval of $\d$-transversality for $\cP_1$ and $\cP_2$.
\end{thm}

See the proof of Theorem 1 at the end of Section 3 in \cite{PS-Bernoulli}, where these statements are deduced from the lemma at
the beginning of Section 3 in \cite{PS-Bernoulli}.

\subsection{Parameters with dimension drop I}\label{sc:Hochman}

Hochman proved that parameters $\l\in[1/2,1)$ with $\dim \nu_\l<1$ can be approximated by algebraic numbers
with high precision.

\begin{thm}[\cite{hochman}*{Theorem 1.9}]\label{th:Hochman1}
Suppose that $\dim\nu_\l<1$ for some $\l\in[1/2,1)$.
Then for every $\theta\in(0,1)$ and for all large enough $n$ (depending on $\l$ and $\theta$), there is a polynomial
$P\neq0$ of degree at most $n$ with coefficients $-1$, $0$ and $1$ such that $|P(\l)|<\theta^n$.
\end{thm}

Since $[0,2^{-1/2}]$ is not an interval of $\d$-transversality for $\cP$ for any $\d>0$,
we will need to borrow a trick from \cite{PS-Bernoulli} and consider trimmed Bernoulli convolutions.
For a parameter $\l\in(0,1)$ we denote by $\wt\nu_\l$ the law of the random variable
$\sum_{n:3\nmid n-2}X_n\l^n$, where
$X_n$ is a sequence of independent random variables taking the values $\pm1$ with equal probability.
These trimmed Bernoulli convolutions satisfy an analogue of Theorem \ref{th:Hochman1} as follows.
We note that $\wt \nu_\l$ are exact dimensional by \cite{feng-hu}, so we can talk about their dimensions
in the same way as in the case of Bernoulli convolutions.
We write $\cQ$ for the set of polynomials whose coefficients are $-1$, $0$ and $1$ and the coefficient of $x^n$
is always $0$ when $3|n-2$.

\begin{thm}\label{th:Hochman2}
Suppose that $\dim\wt\nu_\l<1$ for some $\l\in[2^{-2/3},1)$.
Then for every $\theta\in(0,1)$ and for all large enough $n$ (depending on $\l$ and $\theta$), there is polynomial
$P\neq0\in\cQ$ of degree at most $n$ such that $|P(\l)|<\theta^n$.
\end{thm}

This theorem can be deduced from \cite{hochman}*{Theorems 1.7} in exactly the same way as
\cite{hochman}*{Theorem 1.9}, (which we recalled above in
Theorem \ref{th:Hochman1}).

\subsection{Entropy and Mahler measure}\label{sc:BV1}

Theorem \ref{th:Hochman1} implies that $\dim\nu_\l=1$ for all algebraic parameters $\l\in(1/2,1)$ that are not roots
of polynomials with coefficients $-1$, $0$ and $1$.
Moreover, Hochman \cite{hochman} has the following even more precise result about algebraic parameters.

\begin{thm}\label{th:Hochman3}
Let $\l\in[1/2,1)$ be algebraic.
Then
\[
\dim\nu_\l=\min\Big(1,\frac{h_\l}{\log\l^{-1}}\Big),
\]
where
\[
h_\l=\lim_{N\to \infty}\frac{1}{N}H\Big(\sum_{n=0}^{N-1}X_n\l^n\Big)
\]
and $H(\cdot)$ denotes the Shannon entropy of a discrete random variable.
\end{thm}

See \cite{BV-entropy}*{Section 3.4}, where this is formally deduced from the main result of Hochman \cite{hochman}.

Theorem \ref{th:Hochman3} reduces Question \ref{qu:dim} for algebraic parameters to determining when
$h_\l\ge \log \l^{-1}$ holds.
The quantity $h_\l$ (also called Garsia entropy) received considerable attention recently.
In particular, it was proved by Breuillard and Varj\'u \cite{BV-entropy} that
\[
c\min(\log 2, \log M_\l)\le h_\l\le\min(\log 2, \log M_\l)
\]
for an absolute constant $c>0$.
According to (not rigorous) numerical calculations reported in \cite{BV-entropy}
the constant $c$ can be taken $0.44$.
A recent paper of Akiyama, Feng, Kempton and Persson \cite{AFKP-Bernoulli} gives an
algorithm that allows one to approximate $h_\l$ with arbitrary
precision with a finite computation.
This means that for all algebraic parameters that satisfy $h_\l>\log\l^{-1}$, $\dim\nu_\l=1$ can, in principle, be proved
by finite computation.

The results in \cite{BV-entropy} also allow us to deduce the following.

\begin{thm}\label{th:BV1}
For any $h\in(0,\log 2)$, there is a number $C(h)$ such that
$h_\l\ge h$ for all algebraic numbers $\l$ with $M(\l)\ge C(h)$.
\end{thm}

\begin{rmk}\label{rm:referee}
This theorem is of independent interest.
Following the below proof, one may compute an explicit function $C(h)$ with which the theorem holds
and use it to give new examples of algebraic parameters that make the Bernoulli convolution have
dimension $1$.
\end{rmk}

\begin{proof}
Let $X_0$ be a random variable taking the values $\pm1$ with equal probability and let $G$ be an independent (from $X_0$)
standard Gaussian random variable.
For $a\in\R_{\ge 1}$, we define
\[
\Phi(a)=\sup_{t>0}(H(X_0ta+G)-H(X_0t+G)),
\]
where $H(\cdot)$ is now the differential entropy of an absolutely continuous random variable.

By \cite{BV-entropy}*{Proposition 13}, we have $h_\l\ge \Phi(M(\l))$ for all algebraic numbers $\l$.
It is easy to see that
\[
\lim_{a\to \infty}H(X_0a^{-1/2}+G)=H(G)
\]
and
\[
\lim_{a\to\infty} H(X_0a^{1/2}+G)=\log 2+H(G).
\]
Plugging in $t=a^{-1/2}$ to the definition of $\Phi$, we see that
\[
\lim_{a\to\infty}\Phi(a)=\log 2.
\]
(Here we also used the fact that $\Phi(a)\le \log 2$, which can be seen for example from $h_\l\le\log 2$.)
This proves the claim.
\end{proof}

\subsection{Parameters with dimension drop II}\label{sc:BV2}

Breuillard and Varj\'u also gave approximations for parameters $\l\in[1/2,1)$ with $\dim\nu_\l<1$
by algebraic numbers.
The following is a simplified version of \cite{BV-transcendent}*{Theorem 1}. 

\begin{thm}\label{th:BV2}
Suppose that $\dim\nu_\l<1$ for some $\l\in[1/2,1)$.
Then there are arbitrarily large integers $d$ such that there is an algebraic number $\eta=\eta(d)$
with $\deg\eta\le d$, $\dim\nu_\eta<1$ and
\[
|\l-\eta|\le\exp(-d^2).
\]
\end{thm}

There are a number of differences compared with Theorem \ref{th:Hochman1}.
The approximation is not claimed in Theorem \ref{th:BV2} on all sufficiently large scales,
but only on a (possibly very sparse) sequence of scales.
However, the approximating parameter satisfies $\dim\nu_\eta<1$ and a better estimate for the
distance to $\l$.
Both of these features and the fact that the approximation is available at all (sufficiently large)
scales in Theorem \ref{th:Hochman1} are critically important for the success of our proof of Theorem \ref{th:main}.

\section{Proof of Theorem \ref{th:main}}\label{sc:proof}

We begin by formalizing the statement about lower bounds on the values of polynomials
that we made after Theorem \ref{th:main}.
Recall the notation $\cQ$ from Section \ref{sc:Hochman}.

\begin{prp}\label{pr:separation}
Let $\xi\in(1/2,2^{-1/2})$ be an algebraic number of degree at most $d$
of Mahler measure at most $M$.
Let $n>10 d\log d$ be an integer.
Let $P\neq 0\in\cQ$ be a polynomial of degree at most $n$.
Then $|P(\l)|>(20M)^{-n}$ for all $\l$ satisfying $(5M)^{-n}\le |\l-\xi|\le (5M)^{-n+1}$
provided $d$ is larger than an absolute constant.
If $\xi\in(1/2,2^{-2/3}]$, then the claim also holds for all polynomials $P\neq 0$ of degree at most $n$
with coefficients $-1$, $0$ and $1$.
\end{prp}

\begin{proof}
Suppose to the contrary that $|P(\l)|\le(20M)^{-n}$, where $\l$ is a number in the range specified by the proposition
and $P\neq0$ is a polynomial of degree at most $n$ with coefficients in $-1$, $0$ and  $1$ and $P\in\cQ$ if $\xi>2^{-2/3}$.

If $P(0)\neq 1$ we replace $P$ by $\pm P/x^k$ for a suitable $k$ such that $P(0)=1$ holds.
For this new $P$, we have $P\in\cP$ and moreover $P\in\cP_1\cup\cP_2$ if $\xi>2^{-2/3}$, and
we still have $|P(\l)|\le(10M)^{-n}$.
(Recall the definitions of $\cP$, $\cP_1$ and $\cP_2$ from Section \ref{sc:Solomyak}.)

Let $\d>0$ be a number such that $(0,2^{-2/3})$ is an interval of $\d$-transversality for $\cP$
and $(0,2^{-1/2})$ is an interval of $\d$-transversality for $\cP_1$ and $\cP_2$.
Such a number exists by Theorem \ref{th:Solomyak}.

We first show that $P(\xi)\neq 0$.
If this is not the case, then $P(\xi),P(\l)<\d$ (provided $n$ is large enough), so
we have $P'(t)<-\d$ for all $t$ between $\xi$ and $\lambda$.
Since $|\xi-\l|\ge (5M)^{-n}$, we have $|P(\xi)-P(\l)|>\d(5M)^{-n}$ by the mean value theorem,
a contradiction.

Now we can apply Theorem \ref{th:Garsia} and conclude that
\[
|P(\xi)|\ge (n+1)^{1-d}M(\xi)^{-n}\ge (2M)^{-n}
\]
since $n>10 d\log d$.
Since $|P(\l)|\le(10M)^{-n}$ and $|\l-\xi|\le(5M)^{-n+1}$, we get a contradiction with
$|P'(t)|\le n^2$, which holds for all $t\in[0,1]$.
\end{proof}

\begin{rmk}
After circulating a previous version of this paper, Vesselin Dimitrov pointed out to me that
a variant of Proposition \ref{pr:separation} follows from a result of Mignotte \cite{Mig-larger-degree}, which
gives an estimate for the distance between algebraic numbers if one of the degrees is much larger than the other.
In this version, one may relax the condition $\xi\in(1/2,2^{-2/3})$ at the expense of requiring
$n>d(\log d)^2$.
With this approach, one may avoid using the trimmed version of Bernoulli convolutions $\wt\nu_\l$
later in the proof.
\end{rmk}

We turn to the proof of Theorem \ref{th:main}.
Let $\l\in(1/2,1)$ be a transcendental number and assume to the contrary that $\dim\nu_\l<1$.
We show that $\dim\nu_{\l^k}\le\dim\nu_\l$ for any $k\in\Z_{>0}$.
Note that $\nu_\l=\nu_{\l^k}*\mu$, where $\mu$ is the law of the random variable $\sum_{n:k\nmid n}\pm\l^n$.
Now let $E$ be a Borel set with $\nu_\l(E)>0$.
Then for $t$ belonging to a set of positive $\mu$-measure, we have $\nu_{\l^k}(E-t)>0$, hence
$\dim_H(E)=\dim_H(E-t)\ge \dim\nu_{\l^k}$.
This proves $\dim\nu_\l\ge \dim\nu_{\l^k}$.
Therefore, we can assume $\l<2^{-1/2}$, for otherwise we can replace $\l$ by $\l^k$
for some positive integer $k$ so that $\l^k\in(1/2,2^{-1/2})$.

By Theorem \ref{th:BV2}, there is an arbitrarily large integer $d$, such that there is an algebraic number
$\xi$ of degree at most $d$ with $\dim\nu_\xi<1$ and $|\xi-\l|\le\exp(-d^2)$.
By Theorems \ref{th:Hochman3} and \ref{th:BV1}, there is a number $M$ depending only on $\l$ (but crucially not on $d$)
such that $M(\xi)\le M$.
Choose an integer $n$ such that
\[
(5M)^{-n}\le |\l-\xi|\le (5M)^{-n+1}.
\]

If $\xi\le2^{-2/3}$, we apply Theorem \ref{th:Hochman1} with $\theta=(20M)^{-1}$
to find a polynomial $P\neq0$ of degree at most $n$
with coefficients $-1$, $0$ and $1$
such that $|P(\l)|<(20M)^{-n}$.
Crucially, Theorem \ref{th:Hochman1} holds at every sufficiently large scale $n$.
Now, we have a contradiction with Proposition \ref{pr:separation}.

If $\xi>2^{-2/3}$, then we note that $\dim\wt\nu_\l\le\dim\nu_\l<1$, for $\nu_\l$
is the convolution of $\wt\nu_\l$ with another measure and convolution may only increase the dimension
of measures by the above argument.
Hence, we can apply Theorem \ref{th:Hochman2} and find that there is a polynomial $P\neq 0\in\cQ$
of degree at most $n$ such that $|P(\l)|<(20M)^{-n}$ and we reach a contradiction with
Proposition \ref{pr:separation} again.

\bibliography{bibfile}

\bigskip

\noindent{\sc Centre for Mathematical Sciences,
Wilberforce Road, Cambridge CB3 0WA,
UK}\\
{\em e-mail address:} pv270@dpmms.cam.ac.uk

\end{document}